\documentclass[11pt]{article}
\usepackage{amsthm,amsmath,amssymb,verbatim,fancyhdr}
\usepackage{array,url}
\usepackage{fullpage}

\usepackage{graphicx}

\newtheorem{theorem}{Theorem}[section]

\newtheorem{corollary}[theorem]{Corollary}

\usepackage{geometry}
\usepackage{enumitem}

\begin{document}

\title{Whitney's Theorem for 2-Regular Planar Digraphs}

\author{Dan Archdeacon\\[1mm]
  Department of Mathematics \\ and Statistics\\
  University of Vermont\\
  Burlington, VT, USA
\and
  Matt DeVos\thanks{Supported by a Canadian NSERC Discovery Grant.}\\[1mm]
  Department of Mathematics\\
  Simon Fraser University\\
  Burnaby, BC, Canada\\
  {\tt mdevos@sfu.ca}
\and
  Stefan Hannie\\[1mm]
  Department of Mathematics\\
  Simon Fraser University\\
  Burnaby, BC, Canada\\
  {\tt shannie@sfu.ca}
\and
  Bojan Mohar\thanks{Supported in part by an NSERC Discovery Grant (Canada),
    by the Canada Research Chairs program, and by the
    Research Grant P1--0297 of ARRS (Slovenia).}~\thanks{On leave from:
    IMFM, Department of Mathematics, Ljubljana,
    Slovenia.}\\[1mm]
  Department of Mathematics\\
  Simon Fraser University\\
  Burnaby, BC, Canada\\
  {\tt mohar@sfu.ca}
}

\date{\emph{In memory of Dan Archdeacon, our coauthor and friend.}%
\vspace*{-2em}
}

\maketitle

\begin{abstract}
A digraph is 2\emph{-regular} if every vertex has both indegree and outdegree two.  We define an \emph{embedding} of a 2-regular digraph to be a 2-cell embedding of the underlying graph in a closed surface with the added property that for every vertex~$v$, the two edges directed away from $v$ are not consecutive in the local rotation around $v$.  In other words, at each vertex the incident edges are oriented in-out-in-out.  The goal of this article is to provide an analogue of Whitney's theorem on planar embeddings in the setting of 2-regular digraphs.  In the course of doing so, we note that Tutte's Theorem on peripheral cycles also has a natural analogue in this setting.
\end{abstract}

\section{Introduction}

Although our central subject will be embeddings of certain digraphs in the sphere, we begin by introducing some general terminology for working with embedded (undirected) graphs.  We will use the term \emph{surface} for a closed 2-manifold without boundary.  If $G = (V,E)$ is a graph, an embedding of $G$ in a surface $\mathcal{S}$ is a function $\phi$ which maps $V$ injectively to $\mathcal{S}$ and assigns each edge $uv \in E$ to a simple curve $\phi(uv)$ in $\mathcal{S}$ from $\phi(u)$ to $\phi(v)$ with the added property that the curves associated to distinct edges are internally disjoint.  As usual, we will identify the graph $G$ with its image in $\mathcal{S}$ under $\phi$.  The \emph{faces} are the connected components of the space obtained from $\mathcal{S}$ by deleting $\phi(V) \cup \phi(E)$, and we have a \emph{2-cell} embedding if every face is homeomorphic to an open disc.  In the case of a 2-cell embedding, each face is bounded by a closed walk in $G$, and we call these \emph{facial walks}.

If $\phi_1$ and $\phi_2$ are 2-cell embeddings of $G$ in $\mathcal{S}_1$ and $\mathcal{S}_2$, then we say that $\phi_1$ and $\phi_2$ are \emph{equivalent} if they have the same facial walks.  Throughout we will focus on 2-cell embeddings up to this equivalence, and this permits us to give a purely combinatorial description of an embedded graph by specifying the facial walks. (The embedding may be recovered from this description by treating the graph as a 1-cell complex and adding discs according to the facial walks.) The reader is referred to \cite{MT} for further details. Embeddings of digraphs have been treated previously in \cite{ABJ96, MR1900680, MR2079905, CGH14, Farr13, HLZX09, johnson, sneddon}.

\begin{figure}[htb]
\centerline{\includegraphics[height=4cm]{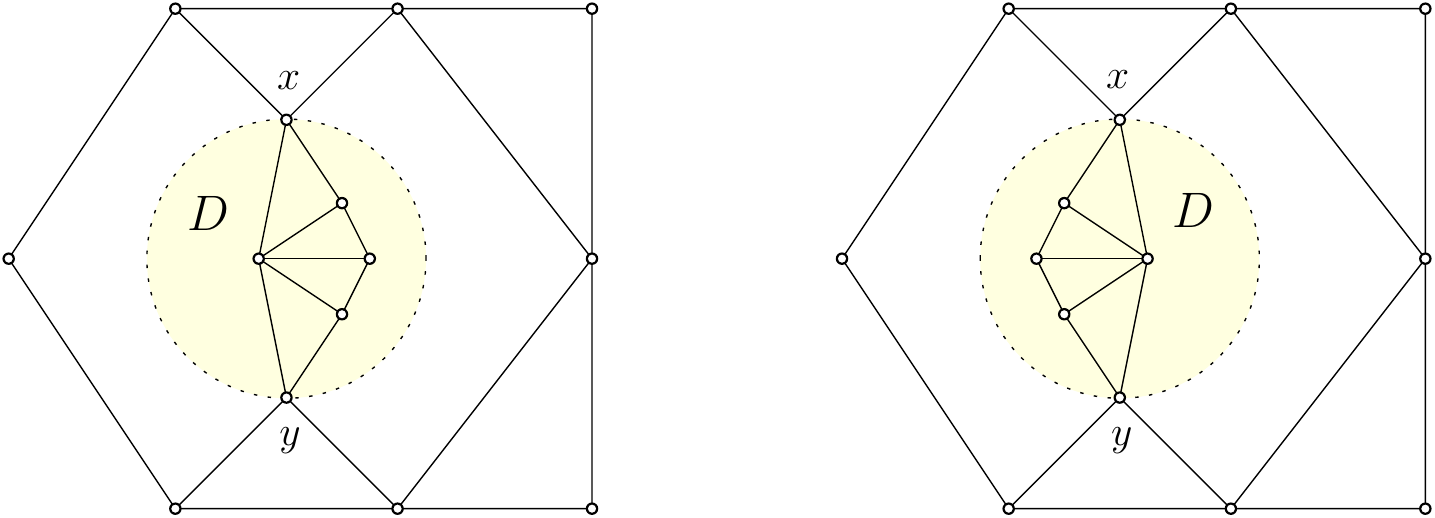}}
\caption{Whitney flip for an undirected embedded graph}
\label{fig:flip undirected}
\end{figure}

Let $\phi$ be a 2-cell embedding of a graph $G$ in a surface $\mathcal{S}$.  Suppose that $\mathcal{D} \subseteq \mathcal{S}$ is homeomorphic to a closed disc, and suppose that the boundary of $\mathcal{D}$ meets $G$ in exactly two points, say $x$ and $y$.  Then there is a homeomorphism $\alpha$ mapping $\mathcal{D}$ to a closed unit disc $D \subseteq \mathbb{R}^2$ where $\alpha(x)$ and $\alpha(y)$ are antipodes of $D$.  Applying the map $\alpha$, followed by the mirror reflection of $D$ through the line containing $\alpha(x)$ and $\alpha(y)$, and then applying $\alpha^{-1}$ gives us a new embedding $\phi'$ of $G$ in the same surface $\mathcal{S}$, and we call this operation a \emph{Whitney flip} of $\mathcal{D}$. See Figure \ref{fig:flip undirected} for an example.

The following theorem shows that for a 2-connected graph embedded in the 2-sphere $S^2$, Whitney flips allow us to move between any two embeddings.

\begin{theorem}[Whitney \cite{Wh33b}]
\label{whitney}
If $\phi_1$ and $\phi_2$ are embeddings of a 2-connected graph~$G$ in the $2$-sphere, then by applying a sequence of Whitney flips, ~$\phi_1$ can be transformed into an embedding equivalent to $\phi_2$.
\end{theorem}

For a 3-connected graph $G$ no meaningful Whitney flip can be performed, so the above theorem immediately yields the following.

\begin{corollary}
\label{maincor}
Any two embeddings of a 3-connected graph in the $2$-sphere are equivalent.
\end{corollary}

In fact, the above corollary is a key step in the proof of Whitney's Theorem.  There is a nice proof of this corollary which comes from the work of Tutte, and we will follow his approach.  For an undirected graph $G$ we say that a cycle $C \subseteq G$ is \emph{peripheral} if $C$ is an induced subgraph and $G - V(C)$ is connected.

\begin{theorem}[Tutte \cite{Tu63}]
\label{tutte}
Every edge in a 3-connected graph is contained in at least two peripheral cycles.
\end{theorem}

Observe that for a connected graph $G$ embedded in $S^2$, every peripheral cycle of $G$ must be the boundary of a face by the Jordan Curve Theorem.  Based on this observation and Theorem \ref{tutte} we have that for a 3-connected graph $G$ embedded in $S^2$, the facial walks are given by the peripheral cycles.  Thus, Theorem \ref{tutte} implies Corollary \ref{maincor}.

Our goal in this article is to prove theorems analogous to Theorem \ref{whitney} and Theorem \ref{tutte} in the setting of 2-regular digraphs.

\section{Theorems of Whitney and Tutte for 2-regular digraphs}

For a 2-regular digraph $H$, we define an \emph{embedding} of $H$ in a surface $\mathcal{S}$ to be an embedding of the underlying graph with the additional property that every face is bounded by a directed walk.  If this holds, then for every vertex $v$, which is the head of two edges $e_1, e_2$ and the tail of two edges $f_1, f_2$, the local configuration is as shown on the left in Figure \ref{vsplit} (i.e. the incident edges go in-out-in-out around $v$). This notion of embedding can be extended to arbitrary Eulerian digraphs and was treated earlier by several authors; see, e.g. \cite{MR1900680, MR2079905} or the paper \cite{ABM} in this issue of the journal.

\begin{figure}[htb]
\centerline{\includegraphics[height=3cm]{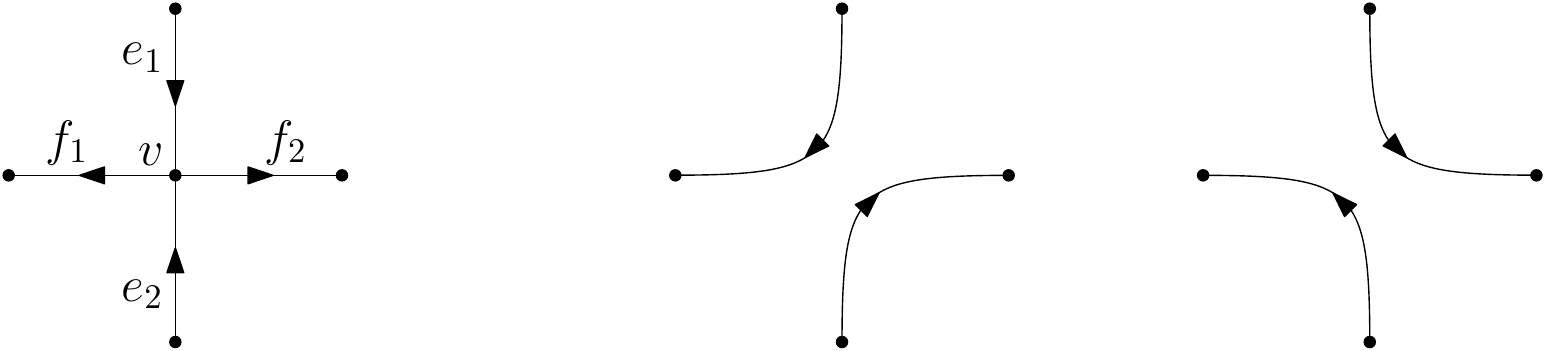}}
\caption{Splitting a vertex}
\label{vsplit}
\end{figure}

To \emph{split} the vertex $v$, we delete $v$ and then add back either a new edge from the tail of $e_i$ to the head of $f_i$ for $i=1,2$ or we add back a new edge from the tail of $e_i$ to the head of $f_{3-i}$ for $i=1,2$.  As shown on the right in Figure \ref{vsplit}, both of the new 2-regular digraphs obtained by splitting $v$ still embed in the surface $\mathcal{S}$ (but the embedding need not be 2-cell any more).  In general, we say that a 2-regular digraph $H'$ which is obtained from $H$ by a sequence of splits is \emph{immersed} in $H$.

Just as any minor of an (ordinary) graph $G$ which embeds in $\mathcal{S}$ also embeds in $\mathcal{S}$, we have that for a 2-regular digraph $H$ which embeds in $\mathcal{S}$, every digraph immersed in $H$ also embeds in $\mathcal{S}$.  In fact, there is a developing theory of 2-regular digraphs under immersion which runs parallel to the classical theory of (undirected) graphs under minors.  The most significant result in this area is due to Johnson \cite{johnson}, who proved an analogue of the Robertson-Seymour structure theorem for 2-regular digraphs which do not immerse a given digraph $H$.  Another theorem of Johnson is the following excluded-immersion theorem for embedding in $S^2$.  Here $C_3^{(2)}$ is the digraph obtained from a directed cycle of length 3 by adding a second edge in parallel with every original one.

\begin{theorem}[Johnson]
A 2-regular digraph $H$ has an embedding in $S^2$ if and only if it does not contain $C_3^{(2)}$ as an immersion.
\end{theorem}

\begin{figure}[htb]
\centerline{\includegraphics[height=4.5cm]{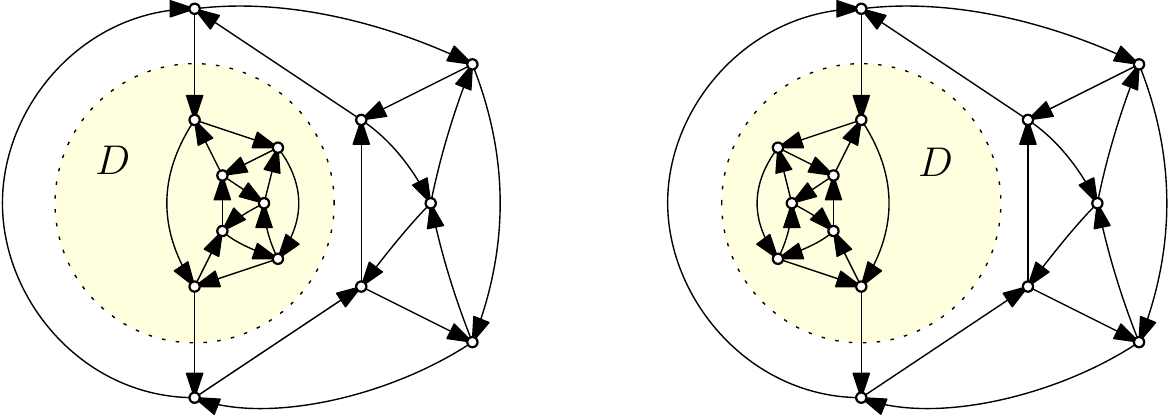}}
\caption{Whitney flip in a 2-regular digraph}
\label{fig:flip directed}
\end{figure}

Our goal here is to prove analogues of theorems of Whitney and Tutte in this new setting.  For a 2-regular digraph $H$ embedded in a surface $\mathcal{S}$, we define a \emph{Whitney flip} to be a Whitney flip of $\mathcal{D}$ in the underlying graph, with the added property that the two points $x,y$ where the boundary of $\mathcal{D}$ meets $H$ are both interior points of edges.  This additional constraint ensures that after performing any Whitney flip, the resulting 2-regular digraph will still be embedded in $\mathcal{S}$ according to our definition.  Our analogue of Whitney's Theorem is as follows.

\begin{theorem}
\label{digraph-whitney}
If $\phi_1$ and $\phi_2$ are embeddings of a connected 2-regular digraph $H$ in the $2$-sphere, then by applying a sequence of Whitney flips, $\phi_1$ can be transformed into an embedding equivalent to $\phi_2$.
\end{theorem}

Let us observe that Theorem \ref{digraph-whitney} does not hold for embeddings of arbitrary Eulerian digraphs. For example, take any 3-connected planar graph and replace each edge $e$ with a pair of oppositely oriented edges $e^+,e^-$ joining the same pair of vertices. This digraph has two embeddings with directed facial walks, but the only way to come from one to the other is to exchange pairs $e^+$ with $e^-$ for all edges at the same time. To get examples without digons, one can replace each digon with a digraph of order 7 (see Figure \ref{fig:flip directed} for a clue).

As before, we have a uniqueness result for digraphs which are suitably connected.  We say that a digraph $H$ is strongly $k$-edge-connected if $H- F$ is strongly connected for every $F \subseteq E(H)$ with $|F| < k$.

\begin{corollary}
\label{ourcor}
Any two embeddings of a strongly 2-edge-connected 2-regular digraph in $S^2$ are equivalent.
\end{corollary}

As was the case with Corollary \ref{maincor}, the above corollary may be proved using a suitable notion of a peripheral cycle.   Consider a 2-regular digraph $H$ embedded in $S^2$. Let $C \subseteq H$ be a directed cycle, and let $\mathcal{D}_1, \mathcal{D}_2$ be the components of $S^2 - C$ (each of which is homeomorphic to a disc).  It follows from the definition of embedding of a 2-regular digraph, that for every vertex $v \in V(C)$, either both edges of $E(H) \setminus E(C)$ incident with $v$ must be contained in the closure of $\mathcal{D}_1$, or both are contained in the closure of $\mathcal{D}_2$.  As a consequence of this, the graph $H - E(C)$ will be disconnected unless either $\mathcal{D}_1$ or $\mathcal{D}_2$ is a face.  In light of this, we define a directed cycle $C$ in a digraph $H$ to be \emph{peripheral} if $H - E(C)$ is strongly connected.  So, as was the case with undirected graphs, in any embedding of a 2-regular digraph in $S^2$, every peripheral cycle must bound a face.  As such, the following theorem implies  Corollary \ref{ourcor}.

\begin{theorem}
\label{digraph-tutte}
Every edge in a strongly 2-edge-connected Eulerian digraph is contained in at least two peripheral cycles.
\end{theorem}

Our proofs of Theorems \ref{digraph-whitney} and \ref{digraph-tutte} are natural analogues of the proofs of Theorems \ref{whitney} and \ref{tutte}.

\section{Proofs}

We begin by proving our analogue of Tutte's peripheral cycles theorem.

\begin{proof}[Proof of Theorem \ref{digraph-tutte}]
Let $e = (u,v)$ be an edge of $H$.  Our first goal will be to find one peripheral cycle through $e$.  To do this, we choose a directed path $P$ from $v$ to $u$ so as to lexicographically maximize the sizes of the components of $H' = H - (E(P) \cup \{e\})$.  That is, we choose the path $P$ so that the largest component of $H'$ is as large as possible, and subject to this the second largest is as large as possible, and so on.

We claim that $H'$ is connected. Suppose (for a contradiction) that $H'$ has components $H_1, H_2, \ldots, H_k$ with $k > 1$ where $H_k$ is a smallest component.  Let $P'$ be the shortest directed path contained in $P$ which contains all vertices in $V(H_k)\cap V(P)$ and suppose the start of $P'$ is the vertex $x$ and the last vertex is $y$.  By construction, $H_k$ must contain both $x$ and $y$.  Furthermore, since $H_k$ is Eulerian, we may choose a directed path $P''$ in $H_k$ from $x$ to $y$.  If there is a component $H_i$ with $i < k$ which contains a vertex in the interior of $P'$, then we get a contradiction to our choice of $P$, since we can reroute the original path along $P''$ instead of $P'$ and get a new path which improves our lexicographic ordering.  Therefore, all vertices in the interior of $P'$ must also be in $H_k$.  However, in this case $H_k \cup P'$ is a subgraph which is separated from the rest of the graph by just two edges, and we have a contradiction to the strong 2-edge-connectivity of $H$.  It follows that $k = 1$, so the cycle $P \cup \{e\}$ is indeed peripheral.

Since the cycle $P \cup \{e\}$ is peripheral, there exists a directed path $Q$ from $v$ to $u$ with $E(Q) \cap E(P) = \emptyset$.  Among all such directed paths $Q$ we choose one so that the unique component of $H - (E(Q) \cup \{e\})$ which contains $P$ is as large as possible, and subject to that we lexicographically maximize the sizes of the remaining components.  By the same argument as above, this choice will result in another peripheral cycle.
\end{proof}

As mentioned before, the above result immediately implies Corollary \ref{ourcor} which asserts that any two embeddings of a strongly 2-edge-connected 2-regular digraph in $S^2$ are equivalent.  Next we will bootstrap this to prove our analogue of Whitney's Theorem.  For a set of vertices $X \subseteq V(H)$ we let $d_H^+(X)$ denote the number of edges with tail in $X$ and head in $V(H) \setminus X$.

\begin{proof}[Proof of Theorem \ref{digraph-whitney}]
We proceed by induction on $|V(H)|$.  If $H$ is strongly 2-edge-connected, then the result follows from Corollary \ref{ourcor}.  Otherwise, we may choose a set $X \subset V(H)$ so that $d_H^+(X) = 1$ and subject to this we choose $X$ minimal.  Let $u \in X$ be the vertex incident with the outgoing edge. Since $H$ is Eulerian, there is precisely one vertex $v \in X$ incident with the edge coming from $V(H) \setminus X$. Let $H'$ be the 2-regular digraph obtained from $H$ by deleting all vertices not in $X$ and then adding the edge $uv$.  This operation may be performed to the embeddings $\phi_1$ and $\phi_2$ of $H$ to obtain embeddings $\phi'_1$ and $\phi_2'$ of $H'$ in $S^2$ where the new edge follows a path from $u$ to $v$ through the vertices in $V(H) \setminus X$.  By the minimality of $X$, the digraph $H'$ is strongly 2-edge-connected (otherwise there would be a set $Y \subset V(H') = X$ with $1 = d_{H'}^+(Y) \ge d_H^+(Y)$ contradicting the choice of $X$).  By Corollary \ref{ourcor}, the embeddings $\phi'_1$ and $\phi'_2$ are equivalent, and they have the same facial walks.

Next we construct a 2-regular digraph $H''$ from $H$ by deleting $X$ and adding a new edge $e''$ joining the end vertices of the two edges between $X$ and $V(H) \setminus X$.  The embeddings $\phi_1$ and $\phi_2$ of $H$ induce embeddings $\phi''_1$ and $\phi_2''$ of $H''$ in $S^2$.  By induction, we may choose a sequence of Whitney flips to apply to $\phi_1''$ to obtain an embedding of $H''$ which is equivalent to $\phi_2''$.  The disks of the flips can be chosen in such a way that the disks never contain the new edge $e''$. Then the same sequence of Whitney flips may also be applied to the original embedding $\phi_1$ of $H$.  After doing so, it follows from our analysis of $H'$ that the resulting embedding is either equivalent to $\phi_2$ or may be made equivalent to $\phi_2$ by performing a Whitney flip on a suitable $\mathcal{D} \subseteq S^2$ which contains precisely the vertices in $X$ and intersects $H$ in the two edges between $X$ and $V(H) \setminus X$.  This completes the proof.
\end{proof}

\subsubsection*{Remark}
This work was done during the sabbatical visit of Dan Archdeacon at the Simon Fraser University from September to December 2013.

\bibliography{mybib}{}
\bibliographystyle{plain}
\end{document}